\documentclass[11pt,reqno]{amsart}

\usepackage{amssymb}
\usepackage{amsmath,amsthm,amsfonts,amssymb,latexsym,mathrsfs,color,hyperref}

\usepackage{tikz}
\newcommand{\vertex}[2]{%
  \fill (#1,0) circle (1.2pt);
  \node[below] at (#1,-0.05) {\small #2};
}

\newcommand{\matchfour}[2]{%
\begin{tikzpicture}[baseline, scale=0.9]
  \foreach \i in {1,2,3,4}
    \vertex{\i}{\i}
  \foreach \a/\b in {#1}
    \draw (\a,0) to[bend left=50] (\b,0);
  \node at (2.5,-0.9) {\small #2};
\end{tikzpicture}
}

\usepackage{xspace}
\usepackage{multirow}
\usepackage{diagbox}
\usepackage{capt-of}
\tikzset{
	level 1/.style = {sibling distance = 1.5cm},
	level 2/.style = {sibling distance = 0.8cm},
    level distance = 0.9 cm
}
\usetikzlibrary {decorations.pathmorphing}
\usetikzlibrary{matrix, positioning}
\tikzstyle{snakeline} = [decorate, decoration={snake, amplitude=.4mm, segment length=2mm}]

\usepackage{tikz-cd}
\usepackage[boxsize=1.3em, centerboxes]{ytableau}

\usepackage{tikz-qtree}
\tikzset{every tree node/.style={minimum width=0.1cm,draw,circle},
         blank/.style={draw=none},
         edge from parent/.style=
         {draw,edge from parent path={(\tikzparentnode) -- (\tikzchildnode)}},
         level distance=0.8cm}

\newtheorem{theorem}{Theorem}
\newtheorem{corollary}[theorem]{Corollary}
\newtheorem{proposition}[theorem]{Proposition}

\newtheorem{lemma}[theorem]{Lemma}
\newtheorem{definition}[theorem]{Definition}
\newtheorem{example}[theorem]{Example}

\newcommand{\md}{\mathcal{D}}

\newcommand{\mm}{\mathcal{M}}

\newcommand{\cda}{{\rm cda\,}}

\newcommand{\drop}{{\rm drop\,}}

\newcommand{\plat}{{\rm plat\,}}

\newcommand{\des}{{\rm des\,}}

\newcommand{\exc}{{\rm exc\,}}

\newcommand{\we}{{\rm wexc\,}}

\newcommand{\cyc}{{\rm cyc\,}}

\newcommand{\fix}{{\rm fix\,}}

\newcommand{\msn}{\mathfrak{S}_n}

\newcommand{\inv}{{\rm inv\,}}

\newcommand{\lrf}[1]{\lfloor #1\rfloor}

\newcommand{\mbn}{{\mathcal B}_n}
\newcommand{\mq}{\mathcal{Q}}
\newcommand{\mqn}{\mathcal{Q}_n}

\newcommand{\asc}{{\rm asc\,}}

\newcommand{\Eulerian}[2]{\genfrac{<}{>}{0pt}{}{#1}{#2}}
\newcommand{\Stirling}[2]{\genfrac{\{}{\}}{0pt}{}{#1}{#2}}
\newcommand{\stirling}[2]{\genfrac{[}{]}{0pt}{}{#1}{#2}}
\newcommand{\arxiv}[1]{\href{http://arxiv.org/abs/#1}{\texttt{arXiv:#1}}}
\title{Eulerian-type polynomials over matchings and matching permutations}
\author[S.-M.~Ma]{Shi-Mei Ma}
\address{School of Mathematics and Statistics, Shandong University of Technology, Zibo 255000, Shandong, P.R. China}
\email{shimeimapapers@163.com (S.-M. Ma)}
\author[Sergey Kitaev]{Sergey Kitaev}
\address{Department of Mathematics and Statistics, University of Strathclyde, 26 Richmond Street, Glasgow G1 1XH, United Kingdom}
\email{sergey.kitaev@strath.ac.uk (S. Kitaev)}
\author[J. Yeh]{Jean Yeh}
\address{Department of Mathematics, National Kaohsiung Normal University, Kaohsiung 82444, Taiwan}
\email{chunchenyeh@nknu.edu.tw (J. Yeh)}
\author[Y.-N. Yeh]{Yeong-Nan Yeh}
\address{College of Mathematics and Physics, Wenzhou University, Wenzhou 325035, P.R. China}
\email{mayeh@alum.sinica.edu.tw (Y.-N. Yeh)}
\begin{document}

\maketitle
\begin{abstract}
Claesson and Linusson [Proc. Am. Math. Soc., 139 (2011), 435--449] observed
that there are $n!$ matchings on $[2n]$ with no left-nestings. 
Inspired by this result, this paper is devoted to exploring a deeper connection between matchings and permutations.
We first discover that a quadruple statistic over 
matchings corresponds to the well known quadruple statistic $\left(\exc,\drop,\fix,\cyc\right)$ 
over permutations, where $\exc,\drop,\fix$ and $\cyc$ are the excedance, drop, fixed point and cycle statistics, respectively. 
By introducing matching permutations,
we provide a symmetric expansion of a five-variable neighbor polynomial of matchings, which encodes a great deal of neighbor information.
As an application, we discover the $e$-positivity of NCA-polynomials, which implies that the left-nesting number, the left-crossing number and the neighbor alignment number are distributed symmetrically over all matchings on $[2n]$. 
We also establish the relationship between the five-variable neighbor polynomials and the trivariate second-order Eulerian polynomials, which generalizes the related results of Claesson and Linusson, Cameron and Killpatrick as well as Chen and Fu.

\bigskip

\end{abstract}
\date{\today}
\tableofcontents
\section{Introduction}

The matchings can be interpreted as fixed-point-free involutions~\cite{Sokal22} and linear chord diagrams~\cite{Alexeev16,Cameron19,Zagier01}.
They have been extensively studied in combinatorics~\cite{Chen07,Corteel07} 
and have interesting applications in algebra~\cite{Chen07,Liu25}, geometry~\cite{BousquetKitaev10,Williams05} and topology~\cite{Stoimenow98,Zagier01}. 
To mention a few, see the adjoint representation of the simple Lie algebra~\cite{Campoamor04}, the linearly independent Vassiliev invariants of knot theory~\cite{Zagier01} and 
the interacting RNA molecules~\cite{Alexeev16}. In this paper,
we shall propose a framework for understanding the relationship between matchings and permutations.

Let $[n]:=\{1,2,\ldots,n\}$.
A {\it matching} on $[2n]$ is a partition of $[2n]$ into $n$ blocks, where each block contains exactly two elements.
Let $\mm_n$ denote the set of matchings on $[2n]$. 
The {\it standard form} of a matching $M\in\mm_n$ is a list of blocks $(i_1,j_1)(i_2,j_2)\cdots(i_n,j_n)$ such that $i_r<j_r$ for all $1\leqslant r\leqslant n$ and $j_1<j_2<\cdots<j_n=2n$. Pictorially, a matching can be represented by an arc diagram where the elements in $[2n]$ are drawn as vertices on a 
horizontal line increasingly labeled and elements of the same block are joined by an arc. 
If $\alpha=(i,j)$ is an arc (or a block) of $M$, we call $i$ the {\it opener} of $\alpha$ and $j$ the {\it closer} of it.
As usual, we always order the arcs with respect to closers, and so the last arc of $M$ is the arc with the closer $2n$.
For example, the three matchings on $[4]$ can be represented by Figure~\ref{fig1}.
\begin{figure}[!ht]\label{fig3}
\renewcommand{\arraystretch}{2}
\begin{center}
\begin{tabular}{c|c|c}
\matchfour{1/2,3/4}{$(1,2)(3,4)$} &
\matchfour{1/3,2/4}{$(1,3)(2,4)$} &
\matchfour{1/4,2/3}{$(2,3)(1,4)$}
\end{tabular}.
\end{center}
\caption{Blocks in each matching are ordered in increasing order of their second elements.}\label{fig1}
\end{figure}

A {\it crossing} (resp. {\it nesting}) of a matching $M\in\mm_n$ is formed by two arcs $(i_1,j_1)$ and $(i_2,j_2)$ such that $i_1<i_2<j_1<j_2$ (resp. $i_1<i_2<j_2<j_1$). 
Following Stoimenow~\cite{Stoimenow98}, if we further require that $i_1+1=i_2$, then such a crossing (resp.~nesting) is called {\it left-crossing} (resp.~{\it left-nesting}). 
Similarly, one can define right-crossing and right-nesting. We say that two arcs $(i_1,j_1)$ and $(i_2,j_2)$ form an {\it alignment} if $i_1<j_1<i_2<j_2$. If we 
further require that $j_1+1=i_2$, then such an alignment is called a {\it neighbor alignment}, which was introduced in~\cite{Chen12}. 

The {\it Catalan number} $C_n=\frac{1}{n+1}\binom{2n}{n}$ counts 
matchings in $\mm_n$ with no crossings (also known as noncrossing matchings), 
see~\cite[Exercise~6.19]{Stanley99}. The {\it Narayana number} $N(n,k)=\frac{1}{n}\binom{n}{k-1}\binom{n}{k}$ counts noncrossing matchings in $\mm_n$ with $k$ blocks of the form $(i,i+1)$, see~\cite[p.~3]{Cameron19} for instance. It is well known that
Eulerian numbers possess many of the same or similar properties as Narayana numbers, see~\cite{Corteel07,Ma19,Petersen15} for instance.
It is natural to investigate the relationship between Eulerian numbers and matchings. 

Let $\operatorname{cr}(M)$ (resp.~$\operatorname{ne}(M)$, $\operatorname{al}(M)$) denote the number of crossings (resp.~nestings, alignments) in $M$.
By adapting the Touchard-Riordan method that encodes matchings by weighted Dyck paths, 
Klazar~\cite{Klazar06} found that
\begin{equation*}\label{Klazar}
\#\{M\in\mm_n: \operatorname{cr}(M)=s, \operatorname{ne}(M)=t\}=\#\{M\in\mm_n: \operatorname{cr}(M)=t, \operatorname{ne}(M)=s\}.
\end{equation*}
In~\cite{Zeng02}, Kasraoui and Zeng proved that 
\begin{equation*}\label{Zeng}
\sum_{M\in\mm_n}x^{\operatorname{cr}(M)}y^{\operatorname{ne}(M)}q^{\operatorname{al}(M)}=
\sum_{M\in\mm_n}x^{\operatorname{ne}(M)}y^{\operatorname{cr}(M)}q^{\operatorname{al}(M)}.
\end{equation*}
Since then, much attention has been devoted to enumerating problems involving crossings and nestings for a large class of combinatorial objects, 
including partial matchings~\cite{Chen12,Levande13}, set partitions~\cite{Chen07,Zeng02,Rubey10,Sokal22} and permutations~\cite{Corteel07,Yen15}. 
Here we list two of them as follows:
 \begin{itemize}
  \item [\rm $(i)$] In order to enumerate totally positive Grassmann cells, 
  Williams~\cite{Williams05} introduced the notation of alignments of a permutation 
  and found a common $q$-analog of the Eulerian numbers, Narayana numbers and binomial coefficients;
     \item [\rm $(ii)$] In an influential work, Corteel~\cite{Corteel07} introduced the notion of crossings and nestings of a permutation, and computed 
the generating function of permutations of $[n]$ with a fixed number of weak excedances, crossings and nestings, 
where a {\it weak excedance} of a permutation is a fixed point or an excedance.
 \end{itemize}

Let $\msn$ be the set of all permutations of $[n]$.
For $\pi\in\msn$, we say that an index $i$ is an {\it excedance} (resp.~{\it drop}, {\it fixed point},~{\it ascent},~{\it descent})
if $\pi(i)>i$ (resp.~$\pi(i)<i$,~$\pi(i)=i$,~$\pi(i)<\pi(i+1)$,~$\pi(i)>\pi(i+1)$).
Let $\exc(\pi)$ (resp.~$\drop(\pi)$, $\fix(\pi)$, $\asc(\pi)$, $\des(\pi)$ and $\cyc(\pi)$) be 
the number of excedances (resp.~drops, fixed points, ascents, descents and cycles) of $\pi$.
It is well known that excedances, drops, ascents and descents are equidistributed over $\msn$. The {\it bivariate Eulerian polynomials} are defined by
\begin{equation}\label{Anxy}
A_n(x,y):=\sum_{\pi\in\msn}x^{\exc(\pi)}y^{\drop(\pi)}=\sum_{\pi\in\msn}x^{\asc(\pi)}y^{\des(\pi)},
\end{equation}
where the last identity can be proved by using {\it Foata's fundamental transformation} (see~\cite{Petersen15}).

An {\it inversion} in $\pi=\pi(1)\pi(2)\cdots\pi(n)\in\msn$ is a pair of entries $\pi(i)$ and $\pi(j)$, where $i<j$ and $\pi(i)>\pi(j)$.
A sequence $\textbf{e}=(e_1,e_2,\ldots,e_n)$ is an {\it inversion sequence} if $0\leqslant e_i<i$. Let $\operatorname{I}_n$ be the set of inversion sequences of length $n$. 
A well known bijection between $\operatorname{I}_n$ and $\msn$ is
defined by $\psi(\pi)=\textbf{e}$, where $e_i=\#\{j\mid 1\leqslant j<i~{\text{and}}~\pi(j)>\pi(i)\}$.
In~\cite[Theorem 2.1]{Claesson11}, Claesson and Linusson found that there is a bijection between $\operatorname{I}_n$ and 
matchings in $\mm_n$ with no left-nestings, see~\cite{Levande13} for further discussion. 
Inspired by Claesson and Linusson's work~\cite{Claesson11}, the paper is devoted to exploring a deeper connection between matchings and permutations.

The structure of the paper is as follows.
In Section~\ref{section2}, we discover that a quadruple statistic on matchings corresponds to the 
quadruple statistic $\left(\exc,\drop,\fix,\cyc\right)$ 
on permutations, i.e., 
$$\sum_{M\in\mm_n}x^{\operatorname{elblock}(M)}y^{\operatorname{olblock}(M)}s^{\operatorname{fixb}(M)}t^{\operatorname{trace}(M)}
=\sum_{\pi\in\msn}(2x)^{\exc(\pi)}(2y)^{\drop(\pi)}(2s)^{\fix(\pi)}\left(\frac{t}{2}\right)^{\cyc(\pi)}.$$
In Section~\ref{section03}, we provide a symmetric expansion of a five-variable neighbor polynomial of matchings.
We also establish the relationship between the neighbor polynomials and the trivariate second-order Eulerian polynomials. 
In particular, by Corollary~\ref{Anxy-Mn},
we get the following result.
\begin{theorem}
Let $\operatorname{lne}$ (resp.~$\operatorname{lcr}$,~$\operatorname{nal}$,~$\operatorname{lrp}$) be the left-nesting (resp.~left-crossing,~neighbor alignment, LR pair) statistic, where a LR pair is a pair of consecutive integer $(i,i+1)$ in the arc diagram of $M\in\mm_n$ such that $i$ is an opener and $i+1$ is a closer ($i$ and $i+1$ possibly in different arcs). The bivariate Eulerian polynomials defined by~\eqref{Anxy} can be interpreted as follows:
\begin{equation*}
A_n(x,y)=\sum_{\substack{M\in \mm_n\\\operatorname{nal}(M)=0}}{x}^{\operatorname{lne}(M)}
{y}^{\operatorname{lcr}(M)}=\sum_{\substack{M\in \mm_n\\\operatorname{lcr}(M)=0}}{x}^{\operatorname{lne}(M)}
{y}^{\operatorname{nal}(M)}=\sum_{\substack{M\in \mm_n\\\operatorname{lne}(M)=0}}{x}^{\operatorname{lcr}(M)}
{y}^{\operatorname{nal}(M)},
\end{equation*}
\begin{equation*}
A_n(x,y)=\sum_{\substack{M\in \mm_n\\\operatorname{lne}(M)=0}}{x}^{\operatorname{lcr}(M)}
{y}^{\operatorname{lrp}(M)-1}=\sum_{\substack{M\in \mm_n\\\operatorname{lcr}(M)=0}}{x}^{\operatorname{lne}(M)}
{y}^{\operatorname{lrp}(M)-1}=\sum_{\substack{M\in \mm_n\\\operatorname{lrp}(M)=1}}{x}^{\operatorname{lne}(M)}
{y}^{\operatorname{lcr}(M)}.
\end{equation*}
\end{theorem}
\section{$(p,q)$-Eulerian polynomials and $(s,t)$-even-odd larger matching polynomials}\label{section2}
\subsection{Definitions and main results}
\hspace*{\parindent}

When $y=1$, the polynomial $A_n(x,y)$ reduces to the {\it Eulerian polynomial} $A_n(x)$, which first appeared in a summation formula due to Euler:
\begin{equation*}\label{Anx-poly-def}
\sum_{k=0}^\infty (k+1)^nx^k=\frac{A_n(x)}{(1-x)^{n+1}}.
\end{equation*}
Let $A_n(x)=\sum_{k=0}^n\Eulerian{n}{k}x^k$, where $\Eulerian{n}{k}$ is called {\it Eulerian number}.
In the past decades, various generalizations of Eulerian polynomials and Eulerian numbers have been pursued by several authors, see~\cite{Branden06,Branden25,Gessel20,Hwang20,Jin23,Savage15} for instance.

Consider the {\it $(p,q)$-Eulerian polynomials}
\begin{equation*}\label{Anxpq-def}
A_n(x,p,q)=\sum_{\pi\in\msn}x^{\exc(\pi)}p^{\fix(\pi)}q^{\cyc(\pi)}.
\end{equation*}
A well known special case is given as follows:
\begin{equation}\label{An11q}
A_n(1,1,q)=\sum_{\pi\in\msn}q^{\cyc(\pi)}=q(q+1)(q+2)\cdots(q+n-1).
\end{equation}
As observed in~\cite{Branden22,Zeng022,Ma24,Ma25,Zeng12,Sokal22}, the polynomials $A_n(x,p,q)$ contain 
a great deal of information about permutations and signed permutations. For example, the {\it $r$-colored Eulerian polynomials} introduced by Steingr\'imsson~\cite{Steingrimsson}
can be expressed as follows (see~\cite[p.~30]{Ma24}):
\begin{equation}\label{Anrx}
A_{n,r}(x)=r^nA_n\left(x,\frac{1+(r-1)x}{r},1\right).
\end{equation}
When $r=1$ and $r=2$, the polynomial $A_{n,r}(x)$ reduces to the types $A$ and $B$ Eulerian polynomials $A_n(x)$ and $B_n(x)$.
Using the {\it exponential formula}, Ksavrelof and Zeng~\cite{Zeng022} found that
\begin{equation}\label{Anxpq-EGF}
\sum_{n=0}^\infty A_n(x,p,q)\frac{z^n}{n!}=\left(\frac{(1-x)\mathrm{e}^{pz}}{\mathrm{e}^{xz}-x\mathrm{e}^{z}}\right)^q.
\end{equation}

An {\it{even-to-odd} block} is a block $(i,j)$ such that $i$ is even and $j$ is odd, where $i<j$. 
Let $a_{n}$ be the number of matchings in $\mm_n$ with no even-to-odd blocks.
Callan~\cite[Proposition~7]{Callan10} found that the exponential generating function of $a_n$ is given as follows:
$$\sum_{n=0}^\infty a_n\frac{z^n}{n!}=\sqrt{\frac{\mathrm{e}^z}{2-\mathrm{e}^z}}.$$
A few years later, Ren~\cite{Ren15} investigated seven sets with the cardinality $a_n$, 
including the set of signed permutations with decreasing blocks all of whose left-to-right-minima have positive signs.
Recently, Sokal and Zeng~\cite[Section~4]{Sokal22} obtained $S$-fractions for perfect matchings by involving right-to-left minima and left-to-right maxima.
They also observed some connections between matchings and permutations, see~\cite[Eq.~(4.5)]{Sokal22} for instance.
As a continuation of these work, we introduce further definitions.
For $M\in\mm_n$, a block $\alpha=(i,j)$ is called 
\begin{itemize}
  \item a {\it fixed block} if $i=2k-1$ and $j=2k$, where $1\leqslant k\leqslant n$;
  \item an {\it odd larger block} if $j$ is odd;
 \item  an {\it odd smaller block} if $i=2k-1$ and $i+1<j$;
 \item an {\it even larger block} if $j=2k$ and $j>i+1$;
 \item an {\it even smaller block} if $i$ is even.
\end{itemize}
Let $\operatorname{fixb}$ (resp.~$\operatorname{olblock}$, $\operatorname{osblock}$, $\operatorname{elblock}$ and $\operatorname{esblock}$) denote the  
fixed block (resp. odd larger block, odd smaller block, even larger block and even smaller block) statistic.

Consider a {\it matching-generation algorithm}:
\begin{itemize}
  \item Let $\psi(M)\in\mm_{n+1}$ be obtained from $M$ by appending the block $(2n+1,2n+2)$;
  \item let $\psi_1(M,(i,j))\in\mm_{n+1}$ be obtained from $M$ by replacing the block $(i,j)$ with two blocks $(i,2n+1)(j,2n+2)$;
 \item let $\psi_2(M,(i,j))\in\mm_{n+1}$ be obtained from $M$ by replacing the block $(i,j)$ with two blocks $(j,2n+1)(i,2n+2)$.
\end{itemize}
It is clear that for any $M'\in\mm_{n+1}$, there is an $M\in\mm_n$ such that either $M'=\psi(M)$, or there is a block $(i,j)$ of $M$ 
such that $M'=\psi_1(M,(i,j))$ or $M'=\psi_2(M,(i,j))$. Conversely, using $\psi^{-1}$, $\psi_1^{-1}$ or 
$\psi_2^{-1}$, we can recover $\mm_n$ from $\mm_{n+1}$ by deleting or contracting the entries $2n+1$ and $2n+2$. 
More precisely, the deletion is defined by $\psi^{-1}$ and the contraction is defined by $\psi_1^{-1}$ or 
$\psi_2^{-1}$.
We say that an index $i$ is a {\it trace index} of $M$ if 
$(i,j)$ is a {\it fixed block} of $M$ or there exists a $k\in [n]$, after deleting or contracting the $2n-2k$ largest entries of $M$, 
the pair $(i,j)$ becomes a {\it fixed block} of the resulting matching.
Clearly, the entry $1$ is always a trace index. Let $\operatorname{trace}(M)$ be the number of trace indices of $M$.

\begin{example}
The matching $({1},3)(2,4)({6},7)({5},8)({9},{10})$ has three trace indices $1$, $5$ and $9$,
while $(2,4)({5},7)(6,8)(3,9)({1},10)$ has two trace indices $1$ and $5$.
\end{example}

The {\it $(s,t)$-even-odd larger matching polynomials} are defined by
$$M_n(x,y,s,t):=\sum_{M\in\mm_n}x^{\operatorname{elblock}(M)}y^{\operatorname{olblock}(M)}s^{\operatorname{fixb}(M)}t^{\operatorname{trace}(M)}.$$
In particular, $M_1(x,y,s,t)=st,~M_2(x,y,s,t)=(st)^2+2txy$ and $M_3(x,y,s,t)=(st)^3+6st^2xy+4txy(x+y)$.
We can now present the first main result of this paper.
\begin{theorem}\label{mainthm01}
\begin{equation}\label{Mnxyst}
M_n(x,y,s,t)=M_n(y,x,s,t)=(2y)^nA_n\left(\frac{x}{y},\frac{s}{y},\frac{t}{2}\right),
\end{equation}
which yields that
\begin{equation}\label{Mnxyst02}
\begin{split}
        2^nA_n(x,p,q)&=\sum_{M\in\mm_n}x^{\operatorname{elblock}(M)}p^{\operatorname{fixb}(M)}(2q)^{\operatorname{trace}(M)}\\
        &=\sum_{M\in\mm_n}x^{\operatorname{olblock}(M)}p^{\operatorname{fixb}(M)}(2q)^{\operatorname{trace}(M)}.
    \end{split}
\end{equation}
\end{theorem}

Combining~\eqref{Anxpq-EGF} and~\eqref{Mnxyst}, we see that 
\begin{equation}\label{MNEGF01}
\sum_{n=0}^\infty M_n(x,y,s,t)\frac{z^n}{n!}=\left(\frac{(y-x)\mathrm{e}^{2sz}}{y\mathrm{e}^{2xz}-x\mathrm{e}^{2yz}}\right)^{\frac{t}{2}}.
\end{equation}
Let $\stirling{n}{k}$ be the {\it signless Stirling numbers of the first kind}.
It follows from~\eqref{An11q} and~\eqref{Mnxyst02} that 
$$\sum_{M\in\mm_n}q^{\operatorname{trace}(M)}=q(q+2)(q+4)(q+6)\cdots(q+2n-2)=\sum_{k=1}^n2^{n-k}\stirling{n}{k}q^k.$$

It is well known that every permutation in $\msn$ can be written as the product of two involutions, 
see~\cite{Petersen13,Yang13}. Hence a natural idea is to write some enumerators over $\msn$ as the convolutions of some enumerators over involutions.
Comparing~\eqref{MNEGF01} with~\eqref{Anxpq-EGF}, we get the following result.
\begin{corollary}\label{convolution}
We have 
$$ 2^nA_n(x,p,q)=\sum_{k=0}^n\binom{n}{k}\sum_{M\in\mm_k}x^{\operatorname{elblock}(M)}p^{\operatorname{fixb}(M)}q^{\operatorname{trace}(M)}
\sum_{M'\in\mm_{n-k}}x^{\operatorname{elblock}(M')}p^{\operatorname{fixb}(M')}q^{\operatorname{trace}(M')}.$$
\end{corollary}

The $p=x$ and $q=1$ case of Corollary~\ref{convolution} was also discussed in~\cite{MaYeh17}.
In~\cite[Corollary~7.4]{Brenti00}, Brenti discovered that 
$A_n(x,1,-1)=-(x-1)^{n-1}$.
Subsequently, Ksavrelof and Zeng~\cite{Zeng022} provided bijective proofs for $A_n(x,1,-1)=-(x-1)^{n-1}$ and 
$A_n(x,0,-1)=-x-x^2-\cdots-x^{n-1}$.
Therefore, the following result is immediate.
\begin{corollary}\label{cor2}
For $n\geqslant 1$, we have 
\begin{align*}
&\sum_{M\in\mm_n}x^{\operatorname{elblock}(M)}(-2)^{\operatorname{trace}(M)}=-2^n(x-1)^{n-1},\\
&\sum_{\substack{M\in\mm_n\\\operatorname{fixb}(M)=0 }}x^{\operatorname{elblock}(M)}(-2)^{\operatorname{trace}(M)}=-2^n(x+x^2+\cdots+x^{n-1}).
\end{align*}
\end{corollary}
\subsection{Proof of Theorem~\ref{mainthm01}}
\hspace*{\parindent}

A {\it context-free grammar} (also known as {\it Chen's grammar}~\cite{Chen93,Dumont96}) $G$ over an alphabet
$V$ is defined as a set of substitution rules replacing a letter in $V$ by a formal function over $V$.
The formal derivative $D_G$ with respect to $G$ satisfies the derivation rules:
$$D_G(u+v)=D_G(u)+D_G(v),~D_G(uv)=D_G(u)v+uD_G(v).$$
Chen~\cite{Chen93} found that if $G=\{a\rightarrow ab,~b\rightarrow b\}$, then $D_G(a)=ab,~D_G(b)=b$, and $D_G^n(a)=a\sum_{k=0}^n\Stirling{n}{k}b^k$,
where $\Stirling{n}{k}$ is {\it the Stirling number of the second kind}. 
In~\cite{Dumont96}, Dumont obtained the grammar for Eulerian polynomials by using a grammatical labeling for circular permutations.
\begin{proposition}[{\cite[Section~2.1]{Dumont96}}]\label{grammar03}
Let $G=\{a\rightarrow ab, b\rightarrow ab\}$.
For any $n\geqslant 1$, one has
\begin{equation*}
D_{G}^n(a)=D_{G}^n(b)=ab^{n}A_n\left(\frac{a}{b}\right).
\end{equation*}
\end{proposition}
Since the work of Dumont, it has been observed that the grammatical calculus 
has remarkable applications in combinatorics, see~\cite{Chen23,Ji25,Ma19,Ma26}. For instance, using a grammatical description of the exponential formula,
Ji~\cite{Ji25} investigated the Euler-Stirling statistics on permutations.

The following result is fundamental. 
\begin{lemma}[{\cite[Lemma~3.12]{Ma24}}]\label{lemma001exc}
For $G=\{I\rightarrow Ipq, p\rightarrow xy, x\rightarrow xy, y\rightarrow xy, q\rightarrow 0\}$, we have
$$D_G^n(I)=I\sum_{\pi\in\msn}x^{\exc(\pi)}y^{\drop(\pi)}p^{\fix(\pi)}q^{\cyc(\pi)}.$$
\end{lemma}

We now give a variant of Lemma~\ref{lemma001exc}.
\begin{lemma}\label{Lemma-matchings}
For $G_1=\{J\rightarrow Jst, s\rightarrow 2ab, a\rightarrow 2ab, b\rightarrow 2ab, t\rightarrow 0\}$, we have
$$D_{G_1}^n(J)=J\sum_{M\in\mm_n}a^{\operatorname{elblock}(M)}b^{\operatorname{olblock}(M)}s^{\operatorname{fixb}(M)}t^{\operatorname{trace}(M)}.$$
\end{lemma}
\begin{proof}
We first introduce a labeling of matchings as follows. Label each
trace index by a subscript label $t$, each fixed block by a superscript label $s$, an even larger block by a superscript label $a$,
an odd larger block by a superscript label $b$ and put a subscript label $J$ at the end.
For example, 
the labeled version of $(2,4)({5},7)(6,8)(3,9)({1},10)(11,12)$ is given below:
\begin{equation}\label{ex01}
{\overbrace{(2,4)}^a\overbrace{(5_t,7)}^b\overbrace{(6,8)}^a\overbrace{(3,9)}^b\overbrace{(1_t,10)}^a{\overbrace{(11_t,12)}^s}}_J.
\end{equation}
The weight of $M$ is given by the product of its labels. For example, the weight of~\eqref{ex01} is $Ja^3b^2st^3$.
The labeled elements in $\mm_1$ and $\mm_2$ can be respectively listed as follows:
$${\overbrace{(1_t,2)}^s}_J,~{\overbrace{(1_t,2)}^s\overbrace{(3_t,4)}^s}_J,~{\overbrace{(1_t,3)}^b\overbrace{(2,4)}^a}_J,~{\overbrace{(2,3)}^b\overbrace{(1_t,4)}^a}_J.$$

Note that $D_{G_1}(J)=Jst$ and $D_{G_2}^2(J)=Js^2t^2+2Jabt$. Hence the result holds for $n=1$ and $n=2$. 
Suppose we get all labeled matchings in $\mm_{n}$, where $n\geqslant 2$. Let
$\widehat{M}$ be obtained from $M\in\mm_{n}$ by using $\psi$, $\psi_1$ or 
$\psi_2$. We have the following possibilities:
\begin{itemize}
  \item [\rm ($c_1$)] if $\widehat{M}$ is obtained from $M$ by appending the block $(2n+1,2n+2)$, then this operation corresponds to the substitution $J\rightarrow Jst$,   
which can be illustrated as follows:
$$\cdots(i,2n)_J\rightarrow {\cdots(i,2n)\overbrace{((2n+1)_t,2n+2)}^s}_J.$$
 \item [\rm ($c_2$)] if $\widehat{M}$ is obtained from $M$ by replacing the block $(i,j)$ with two blocks $(i,2n+1)(j,2n+2)$ or $(j,2n+1)(i,2n+2)$, then 
this operation corresponds to the substitution $s\rightarrow 2ab$, $a\rightarrow 2ab$ or $b\rightarrow 2ab$. More precisely, for a fixed block, the changes of labels can be illustrated as follows: 
$$\cdots\overbrace{((2k-1)_t,2k)}^s\cdots=\left\{\begin{array}{lll}
\cdots\overbrace{((2k-1)_t,2n+1)}^b\overbrace{(2k,2n+2)}^a\cdots,&\\
\cdots\overbrace{(2k,2n+1)}^b\overbrace{((2k-1)_t,2n+2)}^a\cdots.&
\end{array}\right.$$
For an even or odd larger block, the changes of labels can be illustrated as follows: 
$$\cdots\overbrace{(i,2k)}^a\cdots=\left\{\begin{array}{lll}
\cdots\overbrace{(i,2n+1)}^b\overbrace{(2k,2n+2)}^a\cdots,&\\
\cdots\overbrace{(2k,2n+1)}^b\overbrace{(i,2n+2)}^a\cdots;&
\end{array}\right.$$
$$\cdots\overbrace{(i,2k+1)}^b\cdots=\left\{\begin{array}{lll}
\cdots\overbrace{(i,2n+1)}^b\overbrace{(2k+1,2n+2)}^a\cdots,&\\
\cdots\overbrace{(2k+1,2n+1)}^b\overbrace{(i,2n+2)}^a\cdots.&
\end{array}\right.$$
\end{itemize}
Since each case corresponds to an application of a substitution rule in $G_1$, we see that the action of $D_{G_1}$
on the set of labeled matchings in $\mm_{n}$ gives the set of labeled matchings in $\mm_{n+1}$, and so $D_{G_1}^{n+1}(J)$ equals the sum of weights of matchings in $\mm_{n+1}$.
This yields the desired result.
\end{proof}

\noindent{\bf A proof of Theorem~\ref{mainthm01}:}
\begin{proof}
Let $G$ be the grammar given in Lemma~\ref{lemma001exc}. Consider a change of $G$. Setting
$I=J$, $p=2s$, $q=\frac{t}{2}$, $x=2a$ and $y=2b$, then $D_G(J)=Jst$ and $D_G(s)=D_G(a)=D_G(b)=2ab$,
so $G$ can be recast as the grammar $G_1$. 
Substituting $p=2s$, $q=\frac{t}{2}$, $x=2a$ and $y=2b$ in Lemma~\ref{lemma001exc}, we get 
$$M_n(a,b,s,t)=\sum_{\pi\in\msn}(2a)^{\exc(\pi)}(2b)^{\drop(\pi)}(2s)^{\fix(\pi)}\left(\frac{t}{2}\right)^{\cyc(\pi)}.$$
Note that $\drop(\pi)=n-\exc(\pi)-\fix(\pi)$ for $\pi\in\msn$. Then we obtain
\begin{align*}
M_n(a,b,s,t)&=\sum_{\pi\in\msn}(2a)^{\exc(\pi)}(2b)^{n-\exc(\pi)-\fix(\pi)}(2s)^{\fix(\pi)}\left(\frac{t}{2}\right)^{\cyc(\pi)}=(2b)^nA_n\left(\frac{a}{b},\frac{s}{b},\frac{t}{2}\right).
\end{align*}
This completes the proof.
\end{proof}
\subsection{$q$-derangement polynomials and $(p,q)$-Eulerian polynomials of type $B$}
\hspace*{\parindent}

As discussed in~\cite{Branden22,Chen09,Chow09,Hwang20,Jin23,Ma24,Ma25,Mongelli13,Sokal22,Yan26},
by special parametrizations, the $(p,q)$-Eulerian polynomials $A_n(x,p,q)$ can be transformed into several well-studied Eulerian-type polynomials, including
big ascent polynomials, 
$q$-derangement polynomials, $1/k$-Eulerian polynomials and $r$-colored Eulerian polynomials. Therefore, by Theorem~\ref{mainthm01}, we can interpret 
these polynomials as enumerators of matchings. 
In the sequel we collect some examples as illustrations.

A permutation $\pi\in\msn$ is a {\it derangement} if $\fix(\pi)=0$. Let $\md_n$ be the set of all derangements in $\msn$. A well known explicit formula for the derangement numbers $d_n:=\#\md_n$ is given as follows:
$d_n=n!\sum_{i=0}^n\frac{(-1)^i}{i!}$.
Using~\eqref{Mnxyst02}, we get the following result, where $(2n)!!=2^nn!$.
\begin{corollary}
We have
$$\sum_{\substack{M\in\mm_n\\ \operatorname{fixb}(M)=0}}2^{\operatorname{trace}(M)}=(2n)!!\sum_{i=0}^n\frac{(-1)^i}{i!}.$$
\end{corollary}

Let $\cda(\pi):=\#\{i:\pi^{-1}(i)<i<\pi(i)\}$ be the number of {\it cycle double ascents} of $\pi$.
Let $\md_{n,k}:=\{\pi\in \md_n: \cda(\pi)=0,~\exc(\pi)=k\}$. 
Using the theory of continued fractions, Shin and Zeng~\cite[Theorem~11]{Zeng12} obtained that
\begin{equation}\label{dnxq-def}
d_n(x,q)=\sum_{\pi\in\md_n}x^{\exc(\pi)}q^{\cyc(\pi)}=\sum_{k=1}^{\lrf{n/2}}\left(\sum_{\pi\in \md_{n,k}}q^{\cyc(\pi)}\right)x^{k}(1+x)^{n-2k}.
\end{equation}
Combining~\eqref{Mnxyst02} and~\eqref{dnxq-def}, we immediately get the following result.
\begin{proposition}\label{cor-dnk}
We have
$$\sum_{\substack{M\in\mm_n\\ \operatorname{fixb}(M)=0}}x^{\operatorname{elblock}(M)}q^{\operatorname{trace}(M)}=2^nd_n(x,q/2)=\sum_{k=1}^{\lrf{n/2}}\left(\sum_{\pi\in \md_{n,k}}q^{\cyc(\pi)}2^{n-\cyc(\pi)}\right)x^{k}(1+x)^{n-2k}.$$
\end{proposition}
More generally, combining~\eqref{Mnxyst} and~\cite[Theorems~3.4]{Ma24}, we obtain the following result.
\begin{corollary}\label{cor-Mnxyst}
There exist nonnegative integer coefficient polynomials $\gamma_{n,i,j}(t)$ such that 
$$M_n(x,y,s,t)=\sum_{i=0}^ns^i\sum_{j=0}^{\lrf{(n-i)/2}}2^n\gamma_{n,i,j}(t/2)(xy)^j(x+y)^{n-i-2j}.$$
\end{corollary}

Let $\mbn$ be the {\it hyperoctahedral group} of rank $n$.
Elements of $\mbn$ are permutations of $[n]\cup\{-1,-2,\ldots,-n\}$ with the property that $\sigma(-i)=-\sigma(i)$ for all $i\in [n]$.
Let $\sigma=\sigma(1)\sigma(2)\cdots \sigma(n)\in\mbn$.
An {\it excedance} (resp.~{\it fixed point},~{\it singleton}) of $\sigma$ is an index $i\in [n]$ such that $\sigma(|\sigma(i)|)>\sigma(i)$ (resp.~$\sigma(i)=i$,~$\sigma(i)=-i$).
Let $\exc(\sigma)$ (resp.~$\fix(\sigma)$,~$\operatorname{single}(\sigma)$) be the number of excedances (resp.~fixed points, singletons) of $\sigma$. The number of weak excedances of $\sigma$ is defined by $\we(\sigma)=\exc(\sigma)+\operatorname{single}(\sigma)$.
Consider the following {\it $(p,q)$-Eulerian polynomials of type $B$}
 $$B_n(x,p,q):=\sum_{\sigma\in \mbn}x^{\we(\sigma)}p^{\fix(\sigma)}q^{\cyc{\sigma}}.$$
In particular, $B_n(x,1,1)$ reduces to $B_n(x)$ and $B_n(x,0,1)$ reduces to the type $B$ derangement polynomial $d_n^B(x)$, 
see~\cite[Theorem~3.15]{Brenti94} and~\cite{Chen09,Chow09} for details.
The first three $d_n^B(x)$'s are given as follows (see~\cite[p.~4]{Chen09}): 
$d_1^B(x)=x,~d_2^B(x)=4x+x^2$ and $d_3^B(x)=8x+20x^2+x^3$.

According to~\cite[Theorem~5.2]{Ma24}, one has
\begin{equation}\label{Bnxpq01}
B_n(x,p,q)=2^nA_n\left(x,\frac{p+x}{2},q\right).
\end{equation}
Combining this with~\eqref{Mnxyst02}, we get another interpretation of $B_n(x,p,q)$.
\begin{corollary}\label{Bnxpq}
We have
$$B_n(x,p,q)=\sum_{M\in\mm_n}x^{\operatorname{elblock}(M)}(p+x)^{\operatorname{fixb}(M)}2^{\operatorname{trace}(M)-\operatorname{fixb}(M)}q^{\operatorname{trace}(M)}.$$
In particular, the type $B$ derangement polynomials $d_n^B(x)$ can be expressed as
$$d_n^B(x)=\sum_{M\in\mm_n}x^{\operatorname{elblock}(M)+\operatorname{fixb}(M)}2^{\operatorname{trace}(M)-\operatorname{fixb}(M)}.$$
\end{corollary}

Define $$M_n(x,q)=\sum_{M\in\mm_n}x^{\operatorname{elblock}(M)+\operatorname{fixb}(M)}q^{\operatorname{trace}(M)},~
\widetilde{M}_n(x,q)=\sum_{M\in\mm_n}x^{\operatorname{elblock}(M)}q^{\operatorname{trace}(M)}.$$
We end this section by giving a dual result of Corollary~\ref{convolution}.
\begin{proposition}
We have 
$B_n(x,1,q)=\sum_{k=0}^n\binom{n}{k}M_k(x,q)\widetilde{M}_{n-k}(x,q)$.
\end{proposition}
\begin{proof}
By~\eqref{Anxpq-EGF},~\eqref{Mnxyst02} and~\eqref{MNEGF01}, we obtain
\begin{align*}
\sum_{n=0}^\infty M_n(x,q)\frac{z^n}{n!}&=\sum_{n=0}^\infty 2^nA_n\left(x,x,\frac{q}{2}\right)\frac{z^n}{n!}\\
&=\sum_{n=0}^\infty (2x)^nA_n\left(\frac{1}{x},1,\frac{q}{2}\right)\frac{z^n}{n!}\\
&=\sum_{n=0}^\infty x^n\widetilde{M}_n\left(\frac{1}{x},q\right)\frac{z^n}{n!},
\end{align*}
which yields that $M_n(x,q)=x^n\widetilde{M}_n\left(\frac{1}{x},q\right)$.
Using~\eqref{Anxpq-EGF}, we note that $$A_n\left(x,\frac{1+x}{2},q\right)=\sum_{k=0}^n\binom{n}{k}A_k\left(x,x,\frac{q}{2}\right)A_{n-k}\left(x,1,\frac{q}{2}\right).$$
Comparing the above formula with~\eqref{Bnxpq01}, we get the desired convolution formula.  
\end{proof}
\section{Matching permutations, neighbor polynomials and $e$-positivity}\label{section03}
\subsection{Definitions and preliminaries}
\hspace*{\parindent}

For any $M\in\mm_n$, we give a labeling of it as follows. Label the closers of the arcs from left to right with the barred elements 
$\overline{1},\overline{2},\ldots,\overline{n}$, 
and label the openers of the arcs with the corresponding number that they are connected with the unbarred elements $1,2,\ldots,n$.
When we read the labels of $M$ from left to right, 
we get a permutation on $[n]\cup [\overline{n}]=\{1,2,\ldots,n,\overline{1},\overline{2},\ldots,\overline{n}\}$. 
\begin{definition}
A permutation on $[n]\cup [\overline{n}]$ is called a {\it matching permutation} if all barred elements 
$\overline{1},\overline{2},\ldots,\overline{n}$ are arranged in increasing order from left to right, and 
for each $i$, the unbarred $i$ appears to the left of barred $\overline{i}$ (there may exist some elements between $i$ and $\overline{i}$).
\end{definition}

Let $\operatorname{MP}_n$ be the set of matching permutations of $[n]\cup [\overline{n}]$. Clearly, 
$\operatorname{MP}_1=\{1\overline{1}\}$ and $\operatorname{MP}_2=\{1\overline{1}2\overline{2},~12\overline{1}~\overline{2},~21\overline{1}~\overline{2}\}$.
For example, if $\sigma=1\overline{1}2\overline{2}$, we have $\sigma(1)=1,~\sigma(2)=\overline{1},~\sigma(3)=2$ and $\sigma(4)=\overline{2}$.
The correspondence between $\mm_2$ and $\operatorname{MP}_2$ can be represented by Figure~\ref{fig2}.
\begin{figure}[!ht]\label{fig2}
\renewcommand{\arraystretch}{2}
\begin{center}
\begin{tabular}{c|c|c}
\matchfour{1/2,3/4}{$1\overline{1}2\overline{2}$} &
\matchfour{1/3,2/4}{$12\overline{1}~\overline{2}$} &
\matchfour{1/4,2/3}{$21\overline{1}~\overline{2}$}
\end{tabular}.
\end{center}
\caption{The correspondence between $\mm_2$ and $\operatorname{MP}_2$.}\label{fig2}
\end{figure}

The statistics on matching permutations can be translated easily into the corresponding statistics on matchings, and vice versa.
Define
$$I_n(x,y,q):=\sum_{M\in\mm_n}x^{\operatorname{ne}(M)}y^{\operatorname{cr}(M)}q^{\operatorname{al}(M)}.$$
It is reasonable to expect that $I_n(x,y,q)$ enumerates matching permutations
according to some statistics. 
The number of {\it inversions}, {\it co-inversions} and {\it ranks} of $\sigma\in \operatorname{MP}_n$ are defined by
\begin{align*}
\inv(\sigma)&=\#\{(i,j):~i<j,~\sigma(i)>\sigma(j),~{\text{$\sigma(i)$ and $\sigma(j)$ are both unbarred entries}}\};\\
\operatorname{coinv}(\sigma)&=\#\{(i,j):~i<j,~\sigma(i)<\sigma(j),~{\text{$\sigma(i)$ and $\sigma(j)$ are both unbarred entries}}\};\\
\operatorname{rank}(\sigma)&=\#\{(i,j):~i<j,~\sigma(i)<\sigma(j),~{\text{$\sigma(i)$ is a barred entry and $\sigma(j)$ is an unbarred entry}}\}.       
\end{align*}
By adopting the interpretation of matchings as matching permutations, we get the following.
\begin{proposition}
We have 
$$I_n(x,y,q)=\sum_{\sigma\in \operatorname{MP}_n}x^{\operatorname{inv}(M)}y^{\operatorname{coinv}(M)}q^{\operatorname{rank}(M)}.$$
\end{proposition} 

There has been an increasing interest in neighbor information in matchings. We list some of them as follows:
 \begin{itemize}
  \item [\rm $(i)$] Following~\cite{Claesson11}, Stoimenow's diagrams are matchings with no {\it neighbor nestings}, i.e., matchings with neither left-nestings nor 
right-nestings. The Fishburn numbers count Stoimenow's diagrams. Bousquet-M\'elou, Claesson, Dukes and Kitaev~\cite{BousquetKitaev10} constructed bijections between Stoimenow's diagrams and three other structures, including unlabeled $(2+2)$-free posets, ascent sequences and pattern avoiding permutations;
\item [\rm $(ii)$] Chen, Fan and Zhao~\cite{Chen12} found the generating functions for partial matchings avoiding neighbor patterns, including neighbor alignments,
left-nestings and right-nestings;
\item [\rm $(iii)$] Sokal and Zeng~\cite[Section~4.2]{Sokal22} studied crossings and nestings according to whether the element in second position is even or odd, i.e.,
a crossing or nesting $i<j<k<\ell$ is even (resp. odd) if $j$ is even (resp. odd).
  \end{itemize}

In the sequel we shall introduce a neighbor classification.
For $\sigma\in \operatorname{MP}_n$, we define the sets of left-nesting indices, left-crossing indices and near alignment indices by 
\begin{align*}
\operatorname{Lne}(\sigma)&=\{i:~\sigma(i)>\sigma(i+1),~{\text{both $\sigma(i)$ and $\sigma(i+1)$ are unbarred entries}}\};\\
\operatorname{Lcr}(\sigma)&=\{i:~\sigma(i)<\sigma(i+1),~{\text{both $\sigma(i)$ and $\sigma(i+1)$ are unbarred entries}}\};\\
\operatorname{Nal}(\sigma)&=\{i:~{\text{$\sigma(i)$ is a barred entry and $\sigma(i+1)$ is an unbarred entry}}\},
\end{align*}
and let $\operatorname{lne}(\sigma)=\#\operatorname{Lne}(\sigma)$, $\operatorname{lcr}(\sigma)=\#\operatorname{Lcr}(\sigma)$ and $\operatorname{nal}(\sigma)=\#\operatorname{Nal}(\sigma)$.
The sets of {\it RR pair} indices and {\it LR pair} indices of $\sigma$ are respectively defined by 
\begin{align*}
\operatorname{Rrp}(\sigma)&=\{i:~{\text{both $\sigma(i)$ and $\sigma(i+1)$ are barred entries}}\};\\
\operatorname{Lrp}(\sigma)&=\{i:~{\text{$\sigma(i)$ is an unbarred entry and $\sigma(i+1)$ is a barred entry}}\}.
\end{align*}
Set $\operatorname{rrp}(\sigma)=\#\operatorname{Rrp}(\sigma)$ and $\operatorname{lrp}(\sigma)=\#\operatorname{Lrp}(\sigma)$.
Clearly, every index $i\in [2n-1]$ belongs to exactly one of the above five subsets, and we call this 
the {\it neighbor classification}.
Obviously, $\operatorname{lne}(\sigma)$ (resp.~$\operatorname{lcr}(\sigma)$,~$\operatorname{nal}(\sigma)$) 
equals the number of left-nestings (resp.~left-crossings,~neighbor alignments) of the corresponding mathcing. 
Hence we use the same notation for the above five statistics over matching permutations and matchings. 
It should be noted that the definition of LR pair was introduced by Cameron and Killpatrick~\cite{Cameron19}.
Using a bijection between Dyck paths and nonnesting matchings, they~\cite[p.~3]{Cameron19} observed that
$$N(n,k)=\frac{1}{n}\binom{n}{k-1}\binom{n}{k}=\#\{M\in\mm_n: \operatorname{ne}(M)=0, \operatorname{lrp}(M)=k\}.$$

\begin{example}
If $\sigma=21\overline{1}34\overline{2}~\overline{3}~\overline{4}5\overline{5}$, then we have
$$\operatorname{Lne}(\sigma)=\{1\},~\operatorname{Lcr}(\sigma)=\{4\},~\operatorname{Nal}(\sigma)=\{3,8\},
~\operatorname{Rrp}(\sigma)=\{6,7\},~\operatorname{Lrp}(\sigma)=\{2,5,9\}.$$
\end{example}
\subsection{Main results}
\hspace*{\parindent}

The {\it neighbor polynomials} are defined by 
$$C_n(x_1,x_2,x_3,y_1,y_2):=\sum_{\sigma\in \operatorname{MP}_n}{x_1}^{\operatorname{lne}(\sigma)}
{x_2}^{\operatorname{lcr}(\sigma)}{x_3}^{\operatorname{nal}(\sigma)}{y_1}^{\operatorname{rrp}(\sigma)}{y_2}^{\operatorname{lrp}(\sigma)}.$$
In particular, $C_1(x_1,x_2,x_3,y_1,y_2)=y_2$ and $C_2(x_1,x_2,x_3,y_1,y_2)=(x_1+x_2)y_1y_2+x_3y_2^2$.
\begin{lemma}\label{lemma02}
If $G=\{I\rightarrow Ix_1y_1,~x_1\rightarrow x_1x_2y_1,~x_2\rightarrow x_1x_2y_1,~x_3\rightarrow x_1x_3y_1,~y_1\rightarrow x_3y_1y_2,~y_2\rightarrow x_2y_1y_2,~E\rightarrow Ex_3y_2\}$, then we have $D_G^n(Iy_2E)=IEC_{n+1}(x_1,x_2,x_3,y_1,y_2)$.
\end{lemma}
\begin{proof}
For $\sigma\in \operatorname{MP}_n$, we prepend a label $I$ and append a label $E$,
label a left nesting index by $x_1$, a left-crossing index by $x_2$, a near alignment index by $x_3$, a RR pair index by $y_1$ and a LR pair index by $y_2$.
For example, the matching permutation $21\overline{1}34\overline{2}~\overline{3}~\overline{4}5\overline{5}$ can be labeled as
$$^I2_{x_1}1_{y_2}\overline{1}_{x_3}3_{x_2}4_{y_2}\overline{2}_{y_1}~\overline{3}_{y_1}~\overline{4}_{x_3}5_{y_2}\overline{5}^E.$$
The labeled elements in $\operatorname{MP}_1$ and $\operatorname{MP}_2$ can be respectively listed as follows:
$$^I1_{y_2}\overline{1}^E,~^I1_{y_2}\overline{1}_{x_3}2_{y_2}\overline{2}^E,~^I2_{x_1}1_{y_2}\overline{1}_{y_1}~\overline{2}^E,
~^I1_{x_2}2_{y_2}\overline{1}_{y_1}~\overline{2}^E.$$ 
Note that $D_{G}(Iy_2E)=Iy_2E(x_1y_1+x_2y_1+x_3y_2)$. The weight of $\sigma$ is given by the product of its labels. Hence the result holds for $n=1$. 

Let us examine how to generate a matching permutation in $\operatorname{MP}_n$ by first
inserting $\overline{n}$ at the end of a $\sigma\in \operatorname{MP}_{n-1}$ and then inserting the entry $n$ just before $\sigma$ or right after any entry of $\sigma$.  
There are $2n-1$ positions to insert $n$, which can be divided into seven possibilities:
\begin{itemize}
  \item [\rm ($c_1$)] if we insert $n$ just before $\sigma$, then this operation corresponds to the substitution $I\rightarrow Ix_1y_1$,   
which can be illustrated as follows:
$$^I\sigma(1)\sigma(2)\cdots\sigma(2n-3)\overline{n-1}^E\rightarrow ^In_{x_1}\sigma(1)\sigma(2)\cdots\sigma(2n-3)\overline{n-1}_{y_1}~\overline{n}^E;$$
 \item [\rm ($c_2$)] if we insert $n$ right after an element with the left-nesting index, then 
this operation corresponds to the substitution $x_1\rightarrow x_1x_2y_1$, which can be illustrated as follows: 
$$^I\sigma(1)\cdots\sigma(i)_{x_1}\sigma_{i+1}\cdots\sigma(2n-3)\overline{n-1}^E\rightarrow ^I\sigma(1)\cdots\sigma(i)_{x_2}n_{x_1}\sigma_{i+1}\cdots\sigma(2n-3)\overline{n-1}_{y_1}\overline{n}^E;$$
 \item [\rm ($c_3$)] if we insert $n$ right after an element with the left-crossing index, then 
this operation corresponds to the substitution $x_2\rightarrow x_1x_2y_1$, which can be illustrated as follows: 
$$^I\sigma(1)\cdots\sigma(i)_{x_2}\sigma_{i+1}\cdots\sigma(2n-3)\overline{n-1}^E\rightarrow ^I\sigma(1)\cdots\sigma(i)_{x_2}n_{x_1}\sigma_{i+1}\cdots\sigma(2n-3)\overline{n-1}_{y_1}\overline{n}^E;$$
 \item [\rm ($c_4$)] if we insert $n$ right after an element with the near alignment index, then 
this operation corresponds to the substitution $x_3\rightarrow x_1x_3y_1$, which can be illustrated as follows: 
$$^I\sigma(1)\cdots\sigma(i)_{x_3}\sigma_{i+1}\cdots\sigma(2n-3)\overline{n-1}^E\rightarrow ^I\sigma(1)\cdots\sigma(i)_{x_3}n_{x_1}\sigma_{i+1}\cdots\sigma(2n-3)\overline{n-1}_{y_1}\overline{n}^E;$$
 \item [\rm ($c_5$)] if we insert $n$ right after an element with the RR pair index, then 
this operation corresponds to the substitution $y_1\rightarrow x_3y_1y_2$, which can be illustrated as follows: 
$$^I\sigma(1)\cdots\sigma(i)_{y_1}\sigma_{i+1}\cdots\sigma(2n-3)\overline{n-1}^E\rightarrow ^I\sigma(1)\cdots\sigma(i)_{x_3}n_{y_2}\sigma_{i+1}\cdots\sigma(2n-3)\overline{n-1}_{y_1}\overline{n}^E;$$
 \item [\rm ($c_6$)] if we insert $n$ right after an element with the LR pair index, then 
it corresponds to the substitution $y_2\rightarrow x_2y_1y_2$, which can be illustrated as follows: 
$$^I\sigma(1)\cdots\sigma(i)_{y_2}\sigma_{i+1}\cdots\sigma(2n-3)\overline{n-1}^E\rightarrow ^I\sigma(1)\cdots\sigma(i)_{x_2}n_{y_2}\sigma_{i+1}\cdots\sigma(2n-3)\overline{n-1}_{y_1}\overline{n}^E;$$
 \item [\rm ($c_7$)] if we append $n$ at the end of $\sigma$, then 
this operation corresponds to the substitution $E\rightarrow Ex_3y_2$, which can be illustrated as follows: 
$$^I\sigma(1)\cdots\sigma(2n-3)\overline{n-1}^E\rightarrow ^I\sigma(1)\cdots\sigma(2n-3)\overline{n-1}_{x_3}n_{y_2}\overline{n}^E.$$
\end{itemize}
Since each case corresponds to an application of a substitution rule in $G$, we see that the action of $D_{G}$
on the set of labeled matching permutations in $\operatorname{MP}_{n-1}$ gives the set of 
labeled matching permutations in $\operatorname{MP}_{n}$, and so $D_{G}^{n}(Iy_2E)$ equals the sum of weights 
of labeled matching permutations in $\operatorname{MP}_{n+1}$.
This completes the proof.
\end{proof}

A {\it rooted tree} of order $n$ with the vertices labelled $1,2,\ldots,n$, is an increasing tree if the
node labelled $1$ is distinguished as the root, and the labels along
any path from the root are increasing.
An {\it increasing plane tree} is an increasing tree in which the children of each vertex are linearly ordered (from left to right, say).
A 0-1-2-$\cdots$-k {\it increasing plane tree} on $[n]$ is an increasing plane tree with each
vertex with at most $k$ children, see Figure~\ref{Fig03-xy} for instance. The {\it degree} of a vertex in a rooted tree is
the number of its children.
\begin{figure}[ht!]
\begin{center}
\hspace*{\stretch{1}}
\begin{tikzpicture}
\Tree [.\node[label=left:{1}]{};
            \edge; [.\node[label=left:{2}]{};
                \edge; [.\node [label=left:{$3$}] {}; ]]]
        ]
\end{tikzpicture}\hspace*{\stretch{1}}
\begin{tikzpicture}
\Tree [.\node[label=left:{1}]{};
            \edge; [.\node[label=left:{2}]{};]
            \edge; [.\node [label=left:{$3$}] {}; ]]
        ]
\end{tikzpicture}\hspace*{\stretch{1}}
\begin{tikzpicture}
\Tree [.\node[label=left:{1}]{};
            \edge; [.\node[label=left:{3}]{};]
            \edge; [.\node [label=left:{$2$}] {}; ]]
        ]
\end{tikzpicture}.\hspace*{\stretch{1}}
\end{center}
\caption{The 0-1-2 increasing plane trees on $[3]$.}
\label{Fig03-xy}
\end{figure}

We can now present the main result of this section.
\begin{theorem}\label{mainthm2}
For any $n\geqslant 1$, we have the following results:
\begin{itemize}
  \item [\rm ($i$)] The neighbor polynomial has the following expansion with nonnegative coefficients:
\begin{equation}\label{cnxy}
C_{n+1}(x_1,x_2,x_3,y_1,y_2)=y_2\sum_{i+2j+3k=n}\xi_{n;i,j,k}w_1^iw_2^jw_3^k.
\end{equation}
where $w_1=x_1y_1+x_2y_1+x_3y_2,~w_2=x_1x_2y_1^2+x_1x_3y_1y_2+x_2x_3y_1y_2$ and $w_3=x_1x_2x_3y_1^2y_2$. 
Let $\xi_n:=\xi_n(x,y,z)=\sum_{i+2j+3k=n}\xi_{n;i,j,k}x^iy^jz^k$. Then we have the recursion
\begin{equation}\label{xi-recu}
\xi_{n+1}=x\xi_n+2y\frac{\partial}{\partial x}\xi_n+(xy+3z)\frac{\partial}{\partial y}\xi_n+2xz\frac{\partial}{\partial z}\xi_n,
\end{equation}
with the initial conditions $\xi_1(x,y,z)=x$ and $\xi_2(x,y,z)=x^2+2y$. 
  \item [\rm ($ii$)] 
The coefficient $\xi_{n;i,j,k}$ counts the number of 0-1-2-3 increasing plane trees on $[n+1]$ with 
$i$ degree one vertices, $j$ degree two vertices and $k$ degree three vertices. In particular, $\xi_{2;2,0,0}$ counts the leftmost tree in Figure~\ref{Fig03-xy}, 
and $\xi_{2;0,1,0}$ counts the other two trees.
\end{itemize}
\end{theorem}
\begin{proof}
\quad $(i)$
Let $G$ be the grammar given by Lemma~\ref{lemma02}. We consider a change of $G$.
Note that $$D_{G}(Iy_2E)=Iy_2E(x_1y_1+x_2y_1+x_3y_2).$$
Setting $a=Iy_2E$, $b_1=x_1y_1,~b_2=x_2y_1$ and $b_3=x_3y_2$, we see that
$$D_G(a)=a(b_1+b_2+b_3),~D_G(b_1)=b_1(b_2+b_3),~D_G(b_2)=b_2(b_1+b_3),~D_G(b_3)=b_3(b_1+b_2).$$
So we get a new grammar
$G_1=\{a\rightarrow a(b_1+b_2+b_3),~b_1\rightarrow b_1(b_2+b_3),~b_2\rightarrow b_2(b_1+b_3),~b_3\rightarrow b_3(b_1+b_2)\}$, which 
suggests that we may consider the expansion of $D_{G_1}^n(a)$ in terms of the elementary symmetric functions.
So we further set 
$w_1=b_1+b_2+b_3,~w_2=b_1b_2+b_1b_3+b_2b_3,~w_3=b_1b_2b_3$.
It is easy to verify that $G_1$ is transformed into the following grammar
$$G_2=\{a\rightarrow aw_1,~w_1\rightarrow 2w_2,~w_2\rightarrow w_1w_2+3w_3,~w_3\rightarrow 2w_1w_3\}.$$
Note that 
\begin{align*}
D_{G_2}(a)&=aw_1,~
D_{G_2}^2(a)=a(w_1^2+2w_2),~
D_{G_2}^3(a)=a(w_1^3+8w_1w_2+6w_3),\\
D_{G_2}^4(a)&=a(w_1^4+22w_1^2w_2+16w_2^2+42w_1w_3),\\
D_{G_2}^5(a)&=a(w_1^5+52w_1^3w_2+136w_1w_2^2+192w_1^2w_3+180w_2w_3),\\
D_{G_2}^6(a)&=a(w_1^6+114w_1^4w_2+720w_1^2w_2^2+272w_2^3+732w_1^3w_3+2304w_1w_2w_3+540w_3^2).
\end{align*}
Assume that 
\begin{equation}\label{DG2a}
D_{G_2}^n(a)=a\sum_{i+2j+3k=n}\xi_{n;i,j,k}w_1^iw_2^jw_3^k.
\end{equation}
We get 
\begin{align*}
D_{G_2}^{n+1}(a)&=D_{G_2}\left(a\sum_{i+2j+3k=n}\xi_{n;i,j,k}w_1^iw_2^jw_3^k\right)\\
&=a\sum_{i,j,k}\xi_{n;i,j,k}\left(w_1^{i+1}w_2^jw_3^k+2iw_1^{i-1}w_2^{j+1}w_3^k+(j+2k)w_1^{i+1}w_2^jw_3^k+3jw_1^iw_2^{j-1}w_3^{k+1}\right),
\end{align*}
which yields that the coefficients $\xi_{n,i,j,k}$ satisfy the recurrence relation
\begin{equation}\label{xi-recu2}
\xi_{n+1;i,j,k}=(1+j+2k)\xi_{n;i-1,j,k}+2(1+i)\xi_{n;i+1,j-1,k}+3(1+j)\xi_{n;i,j+1,k-1},
\end{equation}
with the initial conditions $\xi_{1;1,0,0}=1$ and $\xi_{1;i,j,k}=0$ for all $(i,j,k)\neq (1,0,0)$. 
Multiplying both sides of~\eqref{xi-recu2} by $x^iy^jz^k$ and summing over all $i,j,k$, we get~\eqref{xi-recu}.
Upon replacing $a=y_2$, $w_1=b_1+b_2+b_3=x_1y_1+x_2y_1+x_3y_2$, $w_2=b_1b_2+b_1b_3+b_2b_3=x_1x_2y_1^2+x_1x_3y_1y_2+x_2x_3y_1y_2$
and $w_3=b_1b_2b_3=x_1x_2x_3y_1^2y_2$ in~\eqref{DG2a}, we immediately get~\eqref{cnxy}. 

\quad $(ii)$ Let 
\begin{equation}\label{gammanijk}
\gamma_n(x,y,z):=\sum_{i+2j+3k=2n+1}\gamma_{n;i,j,k}x^iy^jz^k,
\end{equation} where 
$\gamma_{n;i,j,k}$ is the number of 
0-1-2-3 increasing plane trees on $[n]$ with $k$ leaves, $j$ degree one vertices and $i$ degree two vertices.
Recently, Chen-Fu~\cite[eq.~(4.9)]{Chen22} found that 
$$\gamma_{n;i,j,k}=3(1+i)\gamma_{n-1;i+1,j,k-1}+2(1+j)\gamma_{n-1;i-1,j+1,k-1}+k\gamma_{n-1;i,j-1,k},$$
with the initial conditions $\gamma_{1;0,0,1}=1$ and $\gamma_{n;i,j,k}=0$ if $k\neq 1$. 
Multiplying both sides of this recurrence relation by $x^iy^jz^k$ and summing over all $i,j,k$, we get
\begin{equation}\label{gamma-recu}
\gamma_{n+1}(x,y,z)=3z\frac{\partial}{\partial x}\gamma_n(x,y,z)+2xz\frac{\partial}{\partial y}\gamma_n(x,y,z)+yz\frac{\partial}{\partial z}\gamma_n(x,y,z),
\end{equation}
with the initial conditions $\gamma_1(x,y,z)=z$, $\gamma_2(x,y,z)=yz$ and $\gamma_3(x,y,z)=y^2 z + 2 x z^2$. 
Comparing~\eqref{xi-recu} with~\eqref{gamma-recu}, it is routine to verify that 
\begin{equation}\label{xigamma}
\xi_n(x,y,z)=z^{n+1}\gamma_{n+1}\left(\frac{y}{z},\frac{x}{z},\frac{1}{z}\right).
\end{equation}
Thus we get $\xi_{n;i,j,k}=\gamma_{n+1;j,i,n+1-i-j-k}$, as desired. This completes the proof.
\end{proof}

We define the {\it $\operatorname{NCA}$-polynomials} and {\it $\operatorname{NCR}$-polynomials} as follows:
\begin{align*}
&\operatorname{NCA}_n(x,y,z):=\sum_{\sigma\in \operatorname{MP}_n}{x}^{\operatorname{lne}(\sigma)}
{y}^{\operatorname{lcr}(\sigma)}{z}^{\operatorname{nal}(\sigma)},\\
&\operatorname{NCR}_n(x,y,z):=\sum_{\sigma\in \operatorname{MP}_n}{x}^{\operatorname{lne}(\sigma)}
{y}^{\operatorname{lcr}(\sigma)}{z}^{\operatorname{lrp}(\sigma)-1},
\end{align*}
where $\operatorname{NCA}$ is an abbreviation of left-nesting, left-crossing and neighbor alignment, and
$\operatorname{NCR}$ is an abbreviation of left-nesting, left-crossing and LR pair. Specialization of parameters in~\eqref{cnxy} immediately gives the following result.
\begin{corollary}\label{cor19}
For $n\geqslant 1$, we have $\operatorname{NCA}_{n}(x,y,z)=\operatorname{NCR}_{n}(x,y,z)$ and they are $e$-positive, i.e.,
$\operatorname{NCA}_{n+1}(x,y,z)=\sum_{i+2j+3k=n}\xi_{n;i,j,k}(x+y+z)^i(xy+xz+yz)^j(xyz)^k$.
\end{corollary}

Recall that the {\it bivariate Eulerian polynomials} are defined by $$A_n(x,y)=\sum_{\pi\in\msn}x^{\asc(\pi)}y^{\des(\pi)}.$$
In particular, $A_1(x,y)=1$ and $A_2(x,y)=x+y$.
An index $i\in[n]$ is called a {\it double descent} of $\pi$ if $\pi(i-1)>\pi(i)>\pi(i+1)$, where $\pi(0)=\pi(n+1)=0$.
Foata and Sch\"utzenberger~\cite{Foata70} found that $A_n(x,y)$ is $\gamma$-positive, which has attracted much attention in recent years (see~\cite{Branden08,Zeng12,Yan26}).
\begin{proposition}[\cite{Branden08,Foata70}]\label{Foata70}
One has
$A_n(x,y)=\sum_{i\geqslant 0}\alpha_{n,i}(xy)^i(x+y)^{n-1-2i}$,
where $\alpha_{n,i}$ is the number of permutations $\pi\in \msn$ having no double descents and $\des(\pi)=i$.
\end{proposition}
It is well known that the coefficients $\alpha_{n,i}$ also counts 0-1-2 increasing plane trees on $[n]$ with 
$i$ degree two vertices, see~\cite{Chen23} and~\cite[A101280]{Sloane} for instance. Clearly, a 0-1-2-3 increasing plane tree reduces to a
0-1-2 increasing plane tree if there are no degree three vertices, see Figure~\ref{Fig03-xy} for instance.
It follows from Theorem~\ref{mainthm2} that
$$C_{n}(x,y,0,1,1)=C_{n}(x,0,y,1,1)=C_{n}(0,x,y,1,1)=A_n(x,y).$$ 
Combining this with Corollary~\ref{cor19}, we get the following result.
\begin{corollary}\label{Anxy-Mn}
We have 
\begin{equation*}
A_n(x,y)=\sum_{\substack{\sigma\in \operatorname{MP}_n\\\operatorname{nal}(\sigma)=0}}{x}^{\operatorname{lne}(\sigma)}
{y}^{\operatorname{lcr}(\sigma)}=\sum_{\substack{\sigma\in \operatorname{MP}_n\\\operatorname{lcr}(\sigma)=0}}{x}^{\operatorname{lne}(\sigma)}
{y}^{\operatorname{nal}(\sigma)}=\sum_{\substack{\sigma\in \operatorname{MP}_n\\\operatorname{lne}(\sigma)=0}}{x}^{\operatorname{lcr}(\sigma)}
{y}^{\operatorname{nal}(\sigma)},
\end{equation*}
\begin{equation*}
A_n(x,y)=\sum_{\substack{\sigma\in \operatorname{MP}_n\\\operatorname{lne}(\sigma)=0}}{x}^{\operatorname{lcr}(\sigma)}
{y}^{\operatorname{lrp}(\sigma)-1}=\sum_{\substack{\sigma\in \operatorname{MP}_n\\\operatorname{lcr}(\sigma)=0}}{x}^{\operatorname{lne}(\sigma)}
{y}^{\operatorname{lrp}(\sigma)-1}=\sum_{\substack{\sigma\in \operatorname{MP}_n\\\operatorname{lrp}(\sigma)=1}}{x}^{\operatorname{lne}(\sigma)}
{y}^{\operatorname{lcr}(\sigma)}.
\end{equation*}
\end{corollary}
\subsection{Relationship to trivariate second-order Eulerian polynomials}
\hspace*{\parindent}

According to~\cite{Buckholtz,Carlitz65}, the {\it second-order Eulerian polynomials} can be defined by
$$\left(\frac{x}{1-x}\frac{\mathrm{d}}{\mathrm{d}x}\right)^n\frac{x}{1-x}=\frac{Q_n(x)}{(1-x)^{2r+1}}.$$
Let $[n]_2$ denote the multiset $\{1,1,2,2,\ldots,n,n\}$, where each element in $[n]$ appears twice. 
A {\it Stirling permutation} of order $n$ is a permutation on $[n]_2$ such that 
for each element $i\in[n]$, the values between the two copies of $i$ are larger 
than $i$. Let $\mq_n$ be the set of Stirling permutations of order $n$. 
It is clear that $\mq_1=\{11\}$ and $\mq_2=\{1122,1221,2211\}$. 
As usual, for $\tau=\tau_1\tau_2\cdots \tau_{2n}\in\mqn$, we always set $\tau_0=\tau_{2n+1}=0$. Let 
\begin{align*}
\operatorname{asc}(\tau):=\#\{i:~\tau_i<\tau_{i+1}\},~
\operatorname{plat}(\tau):=\#\{i:~\tau_i=\tau_{i+1}\},~
\operatorname{des}(\tau):=\#\{i:~\tau_i>\tau_{i+1}\}
\end{align*}
denote the number of ascents, plateaux and descents of $\tau$, respectively.
Gessel and Stanley~\cite{Gessel78} discovered that $Q_n(x)=\sum_{\tau\in\mqn}x^{\des(\tau)}$, 
which has various interesting generalizations and variations, see~\cite{Bona08,Haglund12,Hwang20,Ma26,Ma2602,Yan26} for the recent progress on this topic.

Consider the {\it trivariate second-order Eulerian polynomials}
$$Q_n(x,y,z):=\sum_{\tau\in\mqn}x^{\asc(\tau)}y^{\plat(\tau)}z^{\des(\tau)}.$$
Dumont~\cite[p.~317]{Dumont80} found that
\begin{equation}\label{Dumont80}
Q_{n+1}(x,y,z)=xyz\left(\frac{\partial}{\partial x}+\frac{\partial}{\partial y}+\frac{\partial}{\partial z}\right)Q_n(x,y,z),~Q_1(x,y,z)=xyz.
\end{equation}
which implies that $Q_n(x,y,z)$ is symmetric in the variables $x,y$ and $z$. 
The symmetry of $Q_n(x,y,z)$ was rediscovered by Janson~\cite[Theorem~2.1]{Janson08} by constructing an urn model.
Haglund and Visontai~\cite[Section~3.1]{Haglund12} found that $Q_n(x,y,z)$ are stable, which imply that their univariate counter
parts obtained by diagonalization have only real roots.
It should be noted that an equivalent result of~\eqref{Dumont80} is given as follows.
\begin{lemma}[\cite{Chen2102,Ma19}]\label{Lemma-Chen}
If $G=\{x \rightarrow xyz, y\rightarrow xyz, z\rightarrow xyz\}$, then $D_G^n(x)=Q_n(x,y,z)$.
\end{lemma}
For the grammar given in Lemma~\ref{Lemma-Chen}, Chen-Fu~\cite{Chen22} considered the following change of it: 
\begin{equation*}\label{change-grammars02}
\left\{
  \begin{array}{ll}
    u=x+y+z, &  \\
    v=xy+yz+zx, &\\
    w=xyz.
  \end{array}
\right.
\end{equation*}
By Lemma~\ref{Lemma-Chen}, it is routine to verify that
$D_{G}(u)=3w,~D_{{G}}(v)=2uw,~D_{G}(w)=vw$, which yield the grammar
$H=\{u\rightarrow 3w, v\rightarrow 2uw,~w\rightarrow vw\}$.
For $n\geqslant 1$, Chen-Fu~\cite{Chen22} discovered that
\begin{equation*}
Q_n(x,y,z)=D_G^n(x)=D_H^{n-1}(w)=\sum_{i+2j+3k=2n+1}\gamma_{n;i,j,k}u^iv^jw^k,
\end{equation*}
where $\gamma_{n;i,j,k}$ is the same as in~\eqref{gammanijk}.
So we have
\begin{equation}\label{Chen22}
Q_n(x,y,z)=\sum_{i+2j+3k=2n+1}\gamma_{n;i,j,k}(x+y+z)^i(xy+yz+zx)^j(xyz)^k.
\end{equation}
Combining~\eqref{cnxy},~\eqref{xigamma} and~\eqref{Chen22}, we see that 
\begin{align*}
C_{n+1}(x_1,x_2,x_3,y_1,y_2)&=y_2\xi_n(w_1,w_2,w_3)\\
&=y_2w_3^{n+1}\gamma_{n+1}\left(\frac{w_2}{w_3},\frac{w_1}{w_3},\frac{1}{w_3}\right)\\
&=y_2\left(x_1x_2x_3y_1^2y_2\right)^{n+1}Q_{n+1}\left(\frac{1}{x_1y_1},\frac{1}{x_2y_1},\frac{1}{x_3y_2}\right).
\end{align*}
So we get the following result, which provides alternative definitions of $Q_n(x)$ and $Q_n(x,y,z)$.
\begin{theorem}
The neighbor polynomials can be expressed as follows:
$$C_{n}(x_1,x_2,x_3,y_1,y_2)=y_2\left(x_1x_2x_3y_1^2y_2\right)^{n}Q_{n}\left(\frac{1}{x_1y_1},\frac{1}{x_2y_1},\frac{1}{x_3y_2}\right).$$
In particular, we have
\begin{equation}\label{NCA}
\sum_{\sigma\in \operatorname{MP}_n}{x_1}^{\operatorname{lne}(\sigma)}
{x_2}^{\operatorname{lcr}(\sigma)}{x_3}^{\operatorname{nal}(\sigma)}
=(x_1x_2x_3)^{n}\sum_{\tau\in\mqn}\left(\frac{1}{x_1}\right)^{\asc(\tau)}\left(\frac{1}{x_2}\right)^{\plat(\tau)}\left(\frac{1}{x_3}\right)^{\des(\tau)},
\end{equation}
\begin{equation}\label{lnelcrlrp}
\sum_{\sigma\in \operatorname{MP}_n}{x_1}^{\operatorname{lne}(\sigma)}
{x_2}^{\operatorname{lcr}(\sigma)}{y_2}^{\operatorname{lrp}(\sigma)}
=y_2(x_1x_2y_2)^{n}\sum_{\tau\in\mqn}\left(\frac{1}{x_1}\right)^{\asc(\tau)}\left(\frac{1}{x_2}\right)^{\plat(\tau)}\left(\frac{1}{y_2}\right)^{\des(\tau)},
\end{equation}
$$\sum_{\sigma\in \operatorname{MP}_n}{y_1}^{\operatorname{rrp}(\sigma)}{y_2}^{\operatorname{lrp}(\sigma)}
=y_1^{2n}y_2^{n+1}\sum_{\tau\in\mqn}\left(\frac{1}{y_1}\right)^{\asc(\tau)+\plat(\tau)}\left(\frac{1}{y_2}\right)^{\des(\tau)}.$$
\end{theorem}

Claesson and Linusson~\cite[Conjecture~9.2]{Claesson11} conjectured that the distribution of left-nestings 
over matchings is given by the second-order Eulerian triangle, i.e.,
\begin{equation}\label{Qnxmatchings01}
Q_n(x)=x^n\sum_{M\in\mm_n}\left(\frac{1}{x}\right)^{\operatorname{lne}(M)},
\end{equation}
which was proved by Levande~\cite[Appendix A]{Levande13} using an inductive proof. In~\cite{Cameron19}, Cameron and Killpatrick discovered that 
\begin{equation}\label{Qnxmatchings02}
Q_n(x)=x^{n+1}\sum_{M\in\mm_n}\left(\frac{1}{x}\right)^{\operatorname{lrp}(M)}.
\end{equation}   
Therefore, the identity~\eqref{lnelcrlrp} gives a common generalization of~\eqref{Qnxmatchings01} and~\eqref{Qnxmatchings02}.

Recall that the NCA-polynomials are defined by $$\operatorname{NCA}_n(x,y,z)=\sum_{\sigma\in \operatorname{MP}_n}{x}^{\operatorname{lne}(\sigma)}
{y}^{\operatorname{lcr}(\sigma)}{z}^{\operatorname{nal}(\sigma)}.$$
In Corollary~\ref{Anxy-Mn}, we obtain
$\operatorname{NCA}_n(x,y,0)=\operatorname{NCA}_n(x,0,y)=\operatorname{NCA}_n(0,x,y)=A_n(x,y)$.
It follows from~\eqref{NCA} that $$\operatorname{NCA}_n(x,y,z)=(xyz)^nQ_n(1/x,1/y,1/z).$$
Substituting it into~\eqref{Dumont80}, one can easily verify that the NCA polynomials satisfy the recursion
\begin{equation}\label{NCA-recu}
\operatorname{NCA}_{n+1}(x,y,z)=n(x+y+z)\operatorname{NCA}_n(x,y,z)-
\left(x^2\frac{\partial}{\partial x}+y^2\frac{\partial}{\partial y}+z^2\frac{\partial}{\partial z}\right)\operatorname{NCA}_n(x,y,z).
\end{equation}
Below are the NCA polynomials for $n\leqslant 4$:
\begin{align*}
\operatorname{NCA}_{1}(x,y,z)&=1,~\operatorname{NCA}_{2}(x,y,z)=x+y+z,\\
\operatorname{NCA}_{3}(x,y,z)&=x^2 + 4 x y + y^2 + 4 x z + 4 y z + z^2,\\
\operatorname{NCA}_{4}(x,y,z)&=x^3 + 11 x^2 y + 11 x y^2 + y^3 + 11 x^2 z + 36 x y z + 11 y^2 z + 
 11 x z^2 + 11 y z^2 + z^3.
\end{align*}
It would be interesting to provide a bijective proof of~\eqref{NCA-recu}.
\section*{Acknowledgements}
Shi-Mei Ma was supported by the National Natural Science Foundation of China (No. 12071063).
Jean Yeh was supported by the National Science and Technology Council (Grant number: MOST 1132115M017005MY2).
\bibliographystyle{amsplain}

\end{document}